\documentclass[final]{dmtcs-episciences}
\usepackage[utf8]{inputenc}
\usepackage{subfigure}
\usepackage{latexsym}
\usepackage{amsfonts}
\usepackage{amsmath}
\usepackage{amssymb}
\usepackage{amsthm}
\usepackage{enumerate}
\usepackage{caption}
\usepackage{tabu}
\usepackage{float}
\DeclareMathOperator{\inv}{inv}
\usepackage{mathrsfs}

\newtheorem{theorem}{Theorem}[section]
\newtheorem{lemma}{Lemma}[section]

\theoremstyle{definition}
\newtheorem{definition}{Definition}[section]
\newtheorem{remark}{Remark}[section]
\newtheorem{example}{Example}[section]

\author{Colin Defant\affiliationmark{1,2}\thanks{This work was supported by National Science Foundation grant DMS-1358884.}}
\title[Binary Codes and Period-2 Orbits]{Binary Codes and Period-2 Orbits of Sequential Dynamical Systems}

\affiliation{
  Princeton University\\
  University of Florida}

\received{2017-1-19}
\revised{2017-5-30, 2017-9-16}
\accepted{2017-9-20}

\begin{document}
\publicationdetails{19}{2017}{3}{10}{2654}\maketitle

\begin{abstract}
Let $[K_n,f,\pi]$ be the (global) SDS map of a sequential dynamical system (SDS) defined over the complete graph $K_n$ using the update order $\pi\in S_n$ in which all vertex functions are equal to the same function $f\colon\mathbb F_2^n\to\mathbb F_2^n$. Let $\eta_n$ denote the maximum number of periodic orbits of period $2$ that an SDS map of the form $[K_n,f,\pi]$ can have. We show that $\eta_n$ is equal to the maximum number of codewords in a binary code of length $n-1$ with minimum distance at least $3$. This result is significant because it represents the first interpretation of this fascinating coding-theoretic sequence other than its original definition. 
\end{abstract}

\section{Introduction}
Suppose we wish to model a finite system in which objects have various states and update their states in discrete time steps. Moreover, assume that the state to which an object updates depends only on the current state of that object along with the states of other nearby or connected objects. We can capture such a system's behavior with a \emph{graph dynamical system}. A graph dynamical system contains a graph representing the connections between objects, a set of states that the objects can adopt, a collection of functions that model how each individual object updates its state in reaction to its neighbor's states, and a rule determining the scheme by which the objects update their states. 

In a series of papers published between 1999 and 2001, Barrett, Mortveit, and Reidys introduced the notion of a sequential dynamical system (SDS), a graph dynamical system in which vertices update their states sequentially \cite{Barrett1,Barrett2,Barrett3}. Subsequently, several researchers have worked to develop a general theory of SDS (see, for example, \cite{Barrett4,Barrett5,Barrett6,Collina,Garcia1,Macauley,
Macauley2,Reidys1}). The article \cite{Collina} is interesting because it shows how SDS, originally proposed as models of computer simulation, are now being studied in relation with Hecke-Kiselman monoids in algebraic combinatorics. We draw most of our terminology and background information concerning SDS from \cite{Mortveit}, a valuable reference for anyone interested in exploring this field.

In the theory of SDS, the primary focus of many research articles is to count or otherwise characterize periodic orbits in the phase spaces of sequential dynamical systems \cite{Aledo, Barrett8, Barrett5, Barrett7, Tosic, Veliz-Cuba, Wu}. For example, the recent paper \cite{Aledo} studies which periodic orbits can coexist in certain SDS and when certain SDS must necessarily have unique fixed points. In particular, that article shows that analogues of Sharkovsky's theorem from continuous dynamics completely fail to hold for many families of SDS. It is common to analyze the dynamics of sequential dynamical systems defined using classical Boolean functions such and OR, AND, NOR, and NAND (see, for example, \cite{Barrett5, Barrett7}). The article \cite{Wu} focuses more generally on SDS defined using so-called ``bi-threshold" functions. By contrast, we will consider SDS defined using a completely arbitrary update function $f$. As we describe later, this function will be the vertex function for every vertex in the graph (this is possible because we will only consider base graphs that are complete). We now proceed to clarify some of these remarks by establishing some notation and definitions.   

If $v$ is a vertex of a graph $Y$, we let $d(v)$ denote the degree of $v$. We often work in the finite field $\mathbb F_2=\{0,1\}$. In doing so, we let $\overline{x}=1+x$ for any $x\in\mathbb F_2$. For any vector $\vec x=(x_1,x_2,\ldots,x_k)\in \mathbb F_2^k$, let $\inv(\vec x)=(\overline{x_1},\overline{x_2},\ldots,\overline{x_k})$. Furthermore, $\text{id}$ will denote the identity permutation $123\cdots n$ (the length $n$ of the permutation $\text{id}$ will always be clear from context).

An SDS is built from the following parts:
\begin{itemize}
\item An undirected simple graph $Y$ with vertices $v_1,v_2,\ldots,v_n$.
\item A set of states $A$. We will typically use the set of states $A=\mathbb F_2=\{0,1\}$. 
\item A collection of \emph{vertex functions} $\{f_{v_i}\}_{i=1}^n$. Each vertex $v_i$ of $Y$ is endowed with its own vertex function $f_{v_i}\colon A^{d(v_i)+1}\rightarrow A$. 
\item A permutation $\pi\in S_n$. The permutation $\pi$ is known as the \emph{update order}. 
\end{itemize}
Let $q(v)$ denote the state of a vertex $v$. Suppose a vertex $v_i$ has neighbors $v_{j_1},v_{j_2},\ldots,v_{j_{d(v_i)}}$, where $j_1<j_2<\cdots<j_s<i<j_{s+1}<j_{s+2}<\cdots<j_{d(v_i)}$. We let \[X(v_i)=(q(v_{j_1}),q(v_{j_2}),\ldots,q(v_{j_s}),q(v_i),q(v_{j_{s+1}}),q(v_{j_{s+2}}),\ldots,q(v_{j_{d(v_i)}})).\] For example, if the vertex $v_3$ has neighbors $v_1$, $v_4$, and $v_6$, we let $X(v_3)=(q(v_1),q(v_3),q(v_4),q(v_6))$. The vector $(q(v_1),q(v_2),\ldots,q(v_n))$, which lists all of the states of the vertices of $Y$ in the order corresponding to the order of the vertex indices, is known as the \emph{system state} of the SDS. Note that if $Y$ is a complete graph and $v_i$ is any vertex of $Y$, then $X(v_i)$ is equal to the system state of the SDS.

From each vertex function $f_{v_i}$, define the \emph{local update function} $L_{v_i}\colon A^n\rightarrow A^n$ by \[L_{v_i}(x_1,x_2,\ldots,x_n)=(x_1,x_2,\ldots,x_{i-1},f_{v_i}(X(v_i)),x_{i+1},\ldots,x_n).\] Combining these local update functions with the update order $\pi=\pi(1)\pi(2)\cdots\pi(n)$ (we have written the permutation $\pi$ as a word), we obtain the SDS map $F\colon A^n\rightarrow A^n$ given by \[F=L_{v_{\pi(n)}}\circ L_{v_{\pi(n-1)}}\circ\cdots\circ L_{v_{\pi(1)}}.\] We will find it useful to introduce an ``intermediate" SDS map $G_i\colon A^n\rightarrow A^n$ for each $i\in\{1,2,\ldots,n\}$, which we define by \[G_i=L_{v_{\pi(i)}}\circ L_{v_{\pi(i-1)}}\circ\cdots\circ L_{v_{\pi(1)}}.\] Thus, $F=G_n$. We use the convention that $G_0$ denotes the identity map from $A^n$ to $A^n$. The vector $G_i(\vec x)$ represents the system state of the SDS obtained by starting with a system state $\vec x$ and updating only the first $i$ vertices in the update order $\pi$. Once the system updates all $n$ vertices (known as a system update), the new system state is $F(\vec x)$. 

Given any SDS on a graph $Y$ with vertex functions $\{f_{v_i}\}_{i=1}^n$ and update order $\pi$, we denote its SDS map $F$ by the triple $[Y, \{f_{v_i}\}_{i=1}^n, \pi]$. If all of the vertex functions $f_{v_i}$ are equal to the same function $f$, we will simply write $[Y,f,\pi]$ for the 
corresponding SDS map (this situation can only occur if the base graph is regular).

\begin{example} \label{Example1}
Consider the graph $Y$ shown in block A of Figure \ref{Fig1}. 
We define an SDS over $Y$ using the update order $\pi=2413$ and the vertex functions $f_{v_i}$ given by \[f_{v_1}(x_1,x_2,x_3,x_4)=x_1x_3+x_2+x_4,\] \[f_{v_2}(x_1,x_2)=x_1x_2+1,\] \[f_{v_3}(x_1,x_2,x_3)=x_1+x_2+x_3,\] and \[f_{v_4}(x_1,x_2,x_3)=x_1x_2+x_3.\] The initial system state of this SDS is $(0,0,0,1)$, as shown by the blue labels in block A of Figure \ref{Fig1}. Because $\pi(1)=2$, we first update the vertex $v_2$ using the vertex function $f_{v_2}$. We have \[f_{v_2}(X(v_2))=f_{v_2}(q(v_1),q(v_2))=f_{v_2}(0,0)=1,\] so the vertex $v_2$ updates to the new state $1$. This intermediate update is shown in block B of Figure \ref{Fig1}. Another way to understand the transition from block A to block B in the figure is to see that we have changed the system state of the SDS by applying the local update function $L_{v_2}$. Indeed, $L_{v_2}(0,0,0,1)=(0,1,0,1)$. In a similar fashion, we next update vertices $v_4$, $v_1$, and $v_3$. Letting $F=[Y,\{f_{v_i}\}_{i=1}^4,\pi]$, we find that $F(0,0,0,1)=(0,1,1,1)$, as shown in block E of the figure. In other words, through a sequence of local updates, the system update transformed the system state $(0,0,0,1)$ into the new system state $(0,1,1,1)$. 
\end{example} 
\begin{figure}
\begin{center} 
\includegraphics[height=24.5mm]{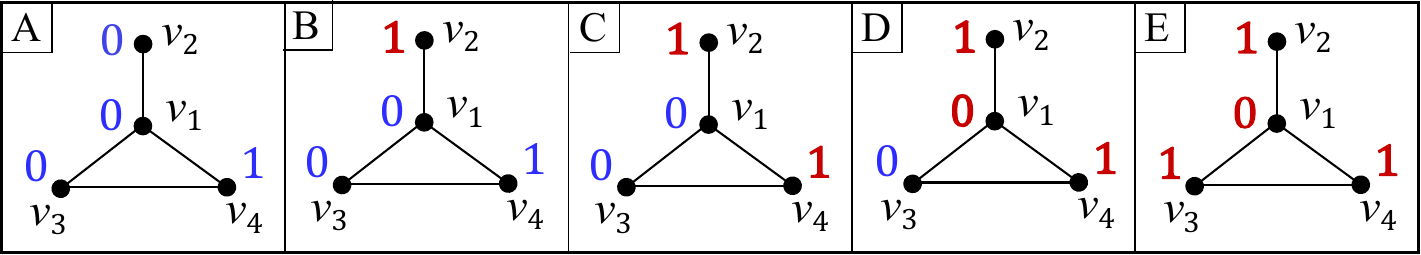}
\end{center}
\caption{\footnotesize{A system update of the SDS of Example \ref{Example1}. Block A shows the inital state of the SDS. Blocks B, C, and D show the intermediate steps of the system update. Block E shows the system state obtained after completing the system update.}} \label{Fig1} 
\end{figure} 

The SDS map $F$ tells us how the states of the vertices of the graph $Y$ change when we update the graph in a sequential manner. A useful tool for visualizing how $F$ acts on the system's states is the \emph{phase space} of the SDS. The phase space, denoted $\Gamma(F)$, is the directed graph with vertex set $V(\Gamma(F))=A^n$ and edge set \[E(\Gamma(F))=\{(\vec x,\vec y)\in A^n\times A^n\colon \vec y=F(\vec x)\}.\] In other words, we draw a directed edge from $\vec x$ to $F(\vec x)$ for each $\vec x\in A^n$. As an example, the phase space of the SDS given in Example \ref{Example1} is shown in Figure \ref{Fig2}. 
\begin{figure}
\begin{center} 
\includegraphics[height=60mm]{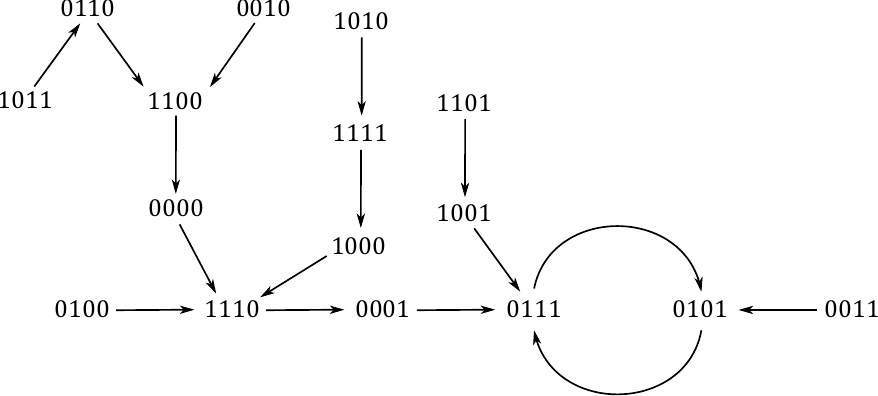}
\end{center}
\caption{\footnotesize{The phase space of the SDS described in Example \ref{Example1}. Note that we have omitted parentheses and commas from the vectors in order to improve the aesthetics of the image. For example, the vector $(1,0,1,1)$ is written as $1011$.}} 
\label{Fig2} 
\end{figure} 
Notice that the phase space shown in Figure \ref{Fig2} has a single $2$-cycle (formed from the vertices $0111$ and $0101$). In general, a phase space of an SDS can have many cycles of various lengths. This leads us to some interesting questions. For example, when is it possible to have a phase space composed entirely of $2$-cycles? Can we find an upper bound on the number of $2$-cycles that can appear in the phase spaces of certain SDS? If we can show, for instance, that certain SDS defined over a graph with $n$ vertices cannot have phase spaces consisting entirely of $2$-cycles, then it will follow that the very natural function $\inv\colon\mathbb F_2^n\rightarrow\mathbb F_2^n$ cannot be the SDS map of any of those SDS. 

In this paper, we study the number of $2$-cycles that can appear in the phase spaces of SDS defined over a complete graph $K_n$ in which all vertex functions are the same. More formally, we give the following definition. 
\begin{definition} \label{Def3}
For a positive integer $n$, let $(\mathbb F_2)^{\mathbb F_2^n}$ be the set of all functions $g\colon\mathbb F_2^n\rightarrow\mathbb F_2$. For each $g\in(\mathbb F_2)^{\mathbb F_2^n}$ and $\pi\in S_n$, let $\eta(g,\pi)$ denote the number of $2$-cycles in the phase space $\Gamma([K_n,g,\pi])$. Define \[\eta_n=\max_{g\in (\mathbb F_2)^{\mathbb F_2^n}}\eta(g,\text{id}).\] 
\end{definition}    
\begin{remark} 
Our decision to use the identity update order $\text{id}$ in the definition of $\eta_n$ in Definition \ref{Def3} stems from a desire for convenience, but we lose no generality in making such a decision because we are working over complete graphs. In other words, \[\eta_n=\max_{g\in (\mathbb F_2)^{\mathbb F_2^n}}\eta(g,\text{id})=\max_{\substack{g\in (\mathbb F_2)^{\mathbb F_2^n} \\ \pi\in S_n}}\eta(g,\pi).\]  
\end{remark} 
In the next section, we reformulate the problem of determining $\eta_n$ in terms of finding the clique number of a certain graph. We then show that $\eta_{n+1}=A(n,3)$, where $A(n,3)$ denotes the maximum number of codewords in a binary code of length $n$ with minimum distance at least $3$. The sequence $A(n,3)$, which is sequence A005864 in Sloane's Online Encyclopedia of Integer Sequences \cite{OEIS}, has been a fascinating and mysterious subject of inquiry in coding theory \cite{MacWilliams, Milshtein, Ostergard, Romanov}. Our result is noteworthy because, to the best of our knowledge, it provides the only known interpretation of this sequence other than its original definition. 

\section{Searching for $\eta_n$} 
If $X$ is a set and $F\colon X\to X$ is a function, we say an element $\vec x$ of $X$ is a periodic point of period $2$ of $F$ if $F^2(\vec x)=\vec x$ and $F(\vec x)\neq\vec x$. Typically, we use the field $\mathbb F_2$ as our set of states. However, in the following lemma, we may use any set of states $A$ so long as $\vert A\vert\geq 2$. 
\begin{lemma} \label{Lem8}
Let $n\geq 2$ be an integer. Let $f\colon A^n\rightarrow A$ be a function, and let $\pi=\pi(1)\pi(2)\cdots\pi(n)\in S_n$ be a permutation. Suppose $\vec x=(x_1,x_2,\ldots, x_n)$ is a periodic point of period $2$ of the SDS map $F=[K_n,f,\pi]$. Write $F(\vec x)=\vec z=(z_1,z_2,\ldots,z_n)$.

\noindent
For each $k\in\{1,2,\ldots,n\}$, we have $x_k\neq z_k$. In particular, if $A=\mathbb F_2$, then $F(\vec x)=\inv(\vec x)$. 
\end{lemma} 
\begin{proof} 
For any $j\in\{1,2,\ldots,n\}$, recall that $G_{j-1}(\vec{x})$ is the system state that results from starting with the initial system state $\vec x$ and then updating the vertices $v_{\pi(1)},v_{\pi(2)},\ldots,v_{\pi(j-1)}$ (in this order). When we update the state of the vertex $v_{\pi(j)}$, the state to which $v_{\pi(j)}$ updates is given by $f(G_{j-1}(\vec x))$. On the other hand, the state to which $v_{\pi(j)}$ updates must be $z_{\pi(j)}$, the state of $v_{\pi(j)}$ in the system state $\vec z$. It follows that $f(G_{j-1}(\vec x))=z_{\pi(j)}$. Note that \[f(G_n(\vec x))=f(F(\vec x))=f(\vec z)=f(G_0(\vec z))=x_{\pi(1)}.\] These same arguments show that $f(G_{j-1}(\vec z))=x_{\pi(j)}$ for any $j\in\{1,2,\ldots,n\}$ and that $f(G_n(\vec z))=z_{\pi(1)}$.

If $i\in\{1,2,\ldots,n-1\}$, then the preceding paragraph tells us that
\[f(G_{i-1}(\vec x))=z_{\pi(i)}, \hspace{1.5cm}f(G_{i-1}(\vec z))=x_{\pi(i)},\] 
\begin{equation} \label{Eq1}
f(G_i(\vec x))=z_{\pi(i+1)},\hspace{1.5cm}f(G_i(\vec z))=x_{\pi(i+1)}. 
\end{equation} 
We also have
\[f(G_{n-1}(\vec x))=z_{\pi(n)},\hspace{1.5cm}f(G_{n-1}(\vec z))=x_{\pi(n)},\] 
\begin{equation} \label{Eq2} 
f(G_n(\vec x))=x_1,\hspace{1.5cm}f(G_n(\vec z))=z_1. 
\end{equation} 
Notice that if $x_{\pi(i)}=z_{\pi(i)}$, then updating the vertex $v_{\pi(i)}$ doesn't ``change'' anything; that is, $G_{i-1}(\vec{x})=G_i(\vec{x})$ and $G_{i-1}(\vec{z}) = G_i(\vec{z}).$  With the help of the equations in \eqref{Eq1} and \eqref{Eq2}, this shows that if $x_{\pi(i)}=z_{\pi(i)}$ for some $i\in\{1,2,\ldots,n\}$, then $x_{\pi(i+1\pmod n)}=z_{\pi(i+1\pmod n)}$. As a consequence, we see that if $x_{\pi(i)}=z_{\pi(i)}$ for some $i\in\{1,2,\ldots,n\}$, then $\vec x=\vec z=F(\vec x)$. However, we are assuming that $\vec x$ is a periodic point of $F$ of period $2$, so $F(\vec x)\neq\vec x$ by definition. Thus, $x_{\pi(i)}\neq z_{\pi(i)}$ for all $i\in\{1,2,\ldots,n\}$, which proves the lemma.  
\end{proof} 

Let $w=w_1w_2\ldots w_k$ be a finite word over the alphabet $\mathbb F_2$. We say that a vector $(x_1,x_2,\ldots,x_n)\in\mathbb F_2^n$ \emph{contains the subsequence $w$} if there exist $i_1,i_2,\ldots,i_k\in\{1,2,\ldots,n\}$ such that $i_1<i_2<\cdots<i_k$ and $x_{i_j}=w_j$ for all $j\in\{1,2,\ldots,k\}$ (this is sometimes expressed by saying that $w$ is a ``scattered subword" of $(x_1,x_2,\ldots,x_n)$). For example, the vector $(x_1,x_2,x_3,x_4,x_5)=(1,0,1,1,0)$ contains the subsequence $100$ because $x_1=1$, $x_2=0$, and $x_5=0$. However, the vector $(1,0,1,1,0)$ does not contain the subsequence $001$. This leads us to the following definition.
\begin{definition} \label{Def1}
Let $w$ be a finite word over the alphabet $\mathbb F_2$. Define $D_n(w)$ to be the set of vectors $\vec x\in\mathbb F_2^n$ such that $\vec x$ contains the subsequence $w$. 
\end{definition}
\begin{definition} \label{Def2}
Let $n\geq 2$ be an integer. Let \[\widehat{\mathbb F}_2^n=\{(x_1,x_2,\ldots,x_n)\in\mathbb F_2^n\colon x_1=0\}.\] Define $\widehat H_n$ to be the undirected simple graph with vertex set \[V(\widehat H_n)=\widehat{\mathbb F}_2^n\] and edge set \[E(\widehat H_n)=\{\{\vec x,\vec y\}\subseteq\widehat{\mathbb F}_2^n\colon\vec x+\vec y\in D_n(101)\}.\]  
\end{definition}
The following two lemmas link the graphs $\widehat H_n$ to our study of $2$-cycles in the phase spaces of SDS defined over complete graphs. 

\begin{lemma} \label{Lem2}
Let $n\geq 2$ be an integer, and let $C=\{\vec x_1,\vec x_2,\ldots, \vec x_k\}$ be a clique of order $k$ of $\widehat H_n$. There exists a map $f\colon \mathbb F_2^n\rightarrow \mathbb F_2$ such that each $\vec x_i\in C$ is in a $2$-cycle of $\Gamma([K_n,f,\text{\emph{id}}])$. Moreover, no two distinct $\vec{x}_i, \vec{x}_j$ are contained in the same 2-cycle of $\Gamma([K_n,f,\text{\emph{id}}])$.
\end{lemma} 
\begin{proof} 
For each $i\in\{1,2,\ldots,k\}$, let $\vec x_i=(a_{i1},a_{i2},\ldots,a_{in})$, where $a_{i1}=0$ by the definition of $\widehat H_n$. Define the map $f\colon \mathbb F_2^n\rightarrow \mathbb F_2$ as follows:
\begin{itemize}
\item If there are some $i, \ell$ such that \[\vec \alpha=(\overline{a_{i1}},\overline{a_{i2}},\ldots,\overline{a_{i\ell}},a_{i(\ell+1)},a_{i(\ell+2)},\ldots,a_{in}),\] then let $f(\vec\alpha)=\overline{a_{i(\ell+1)}}$.
\item If there are some $i, \ell$ such that \[\vec \alpha=(a_{i1},a_{i2},\ldots,a_{i\ell},\overline{a_{i(\ell+1)}},\overline{a_{i(\ell+2)}},\ldots,\overline{a_{in}}),\] then let $f(\vec \alpha)=a_{i(\ell+1)}$. 
\item Otherwise, let $f(\vec \alpha)=0$.
\end{itemize}  

We first need to show that $f$ is well-defined. To do so, we show that for any $i,j\in\{1,2,\ldots,k\}$ and any $\ell,m\in\{0,1,\ldots,n-1\}$ with $i\neq j$ or $\ell\neq m$, we have \begin{equation} \label{Eq4}
(\overline{a_{i1}},\ldots,\overline{a_{i\ell}},a_{i(\ell+1)},\ldots,a_{in})\neq(\overline{a_{j1}},\ldots,\overline{a_{jm}},a_{j(m+1)},\ldots,a_{jn}),
\end{equation} 
\begin{equation} \label{Eq5} 
(\overline{a_{i1}},\ldots,\overline{a_{i\ell}},a_{i(\ell+1)},\ldots,a_{in})\neq(a_{j1},\ldots,a_{jm},\overline{a_{j(m+1)}},\ldots,\overline{a_{jn}}),
\end{equation} 
and 
\begin{equation}\label{Eq13}
(a_{i\ell},\ldots,a_{i\ell},\overline{a_{i(\ell+1)}},\ldots,\overline{a_{in}}) \neq (a_{j1},\ldots,a_{jm},\overline{a_{j(m+1)}},\ldots,\overline{a_{jn}}).
\end{equation}
This will show that we have not accidentally defined $f(\vec \alpha)=0$ and $f(\vec \alpha)=1$ for the same vector $\vec \alpha$. 

It is easy to see that the lack of equality must hold in (3), (4), and (5) if $i=j$ and $\ell\neq m$. Thus, we may assume $i\neq j$. If equality holds in (3) or (5), then $x_i+x_j =(0,0,\ldots,0,1,\ldots,1,0,\ldots, 0)\not\in D_n(101)$, which is a contradiction. Similarly, if equality holds in (4), then it follows from the fact that $a_{i1}=a_{j1}=0$ that $x_i+x_j=(0,0,\ldots,0,1,1,\ldots,1)\not\in D_n(101)$. This is again a contradiction. Thus, $f$ is well-defined.

Let $F=[K_n,f,\text{id}]$. We now show that for each $i\in\{1,2,\ldots,k\}$, the vector $\vec x_i$ is in a $2$-cycle of $\Gamma(F)$. Choose some $i\in\{1,2,\ldots,k\}$. It follows from the definition of $f$ that for each $\ell\in\{1,2,\ldots,n\}$ \[G_\ell(\vec x_i)=(\overline{a_{i1}},\ldots,\overline{a_{i\ell}},a_{i(\ell+1)},\ldots,a_{in})\] and \[G_\ell(\inv(\vec x_i))=G_\ell(\overline{a_{i1}},\overline{a_{i2}},\ldots,\overline{a_{in}})=(a_{i1},\ldots,a_{i\ell},\overline{a_{i(\ell+1)}},\ldots,\overline{a_{in}}).\] In particular, $F(\vec x_i)=\inv(\vec x_i)\neq \vec x_i$ and $F^2(\vec x_i)=F(\inv(\vec x_i))=\vec x_i$. In other words, $\vec x_i$ is in a $2$-cycle of $\Gamma(F)$. 

Finally, choose some distinct $i,j\in\{1,2,\ldots,k\}$. It is easy to see that the vectors $\vec x_i$ and $\vec x_j$ are in distinct $2$-cycles of $\Gamma(F)$. Since $\vec x_i\neq \vec x_j$, the only way the vectors $\vec x_i$ and $\vec x_j$ could be in the same $2$-cycle of $\Gamma(F)$ is if $F(\vec x_i)=\vec x_j$. We have just shown that $F(\vec x_i)=\inv(\vec x_i)$, so $\vec x_i$ could only be in the same $2$-cycle as $\vec x_j$ if $\vec x_j=\inv(\vec x_i)$. However, this is impossible because $\vec x_i$ and $\vec x_j$ have the same first coordinate (namely, $0$).
\end{proof}

\begin{lemma} \label{Lem3}
Let $n\geq 2$ be an integer, and let $\vec x$ and $\vec y$ be distinct nonadjacent vertices of $\widehat H_n$. Let $f\colon \mathbb F_2^n\rightarrow \mathbb F_2$ be a function. If $\vec x$ is in a $2$-cycle of the phase space $\Gamma([K_n,f,\text{\emph{id}}])$, then $\vec y$ is not in a $2$-cycle of $\Gamma([K_n,f,\text{\emph{id}}])$.\end{lemma} 
\begin{proof} 
Suppose, by way of contradiction, that $\vec x$ and $\vec y$ are both in $2$-cycles of $\Gamma([K_n,f,\text{id}])$. Let $F=[K_n,f,\text{id}]$. Let $\vec x=(a_1,a_2,\ldots,a_n)$ and $\vec y=(b_1,b_2,\ldots,b_n)$. Because $\vec x\neq\vec y$, we may let $r$ be the smallest element of $\{1,\ldots,n\}$ such that $a_r\neq b_r$. Similarly, we may let $s$ be the largest element of $\{1,\ldots,n\}$ such that $a_s\neq b_s$. Note that $r\geq 2$ because $a_1=b_1=0$. Because $\vec x$ and $\vec y$ are not adjacent in $\widehat H_n$, the vector $\vec x+\vec y=(a_1+b_1,\ldots,a_n+b_n)$ does not contain the subsequence $101$. This implies that \[a_i+b_i=\begin{cases} 1, & \mbox{if } r\leq i\leq s; \\ 0, & \mbox{otherwise}. \end{cases}\] In other words, $\vec y=(a_1,\ldots,a_{r-1},\overline{a_r},\ldots,\overline{a_s},a_{s+1},\ldots,a_n)$. 
For the sake of convenience, let $a_{n+1}=\overline{a_1}$. 

Because we are assuming that each of the vectors $\vec x$ and $\vec y$ is in a $2$-cycle of $\Gamma(F)$, we know from Lemma \ref{Lem8} that $F(\vec x)=\inv(\vec x)$ and $F(\vec y)=\inv(\vec y)$. Hence, for each $\ell\in\{1,\ldots,n\}$, \[G_\ell(\vec x)=(\overline{a_1},\ldots,\overline{a_\ell},a_{\ell+1},\ldots,a_n)\] and \[G_\ell(\vec y)=(\overline{b_1},\ldots,\overline{b_\ell},b_{\ell+1},\ldots,b_n).\] It follows that \[f(\overline{a_1},\ldots,\overline{a_s},a_{s+1}\ldots,a_n)=f(G_s(\vec x))=\overline{a_{s+1}},\] but also \[f(\overline{a_1},\ldots,\overline{a_s},a_{s+1}\ldots,a_n)=f(\overline{b_1},\ldots,\overline{b_{r-1}},b_r,\ldots,b_n)=f(G_{r-1}(\vec y))=\overline{b_r}=a_r.\] Therefore, $a_r=\overline{a_{s+1}}$. Similarly, we have \[f(\overline{a_1},\ldots,\overline{a_{r-1}},a_r\ldots,a_n)=f(G_{r-1}(\vec x))=\overline{a_r}\] and \[f(\overline{a_1},\ldots,\overline{a_{r-1}},a_r\ldots,a_n)=f(\overline{b_1},\ldots,\overline{b_s},b_{s+1},\ldots,b_n)=f(G_s(\vec y))=\overline{b_{s+1}}=\overline{a_{s+1}}.\] This implies that $\overline{a_r}=\overline{a_{s+1}}=a_r$, which is a contradiction. 
\end{proof} 

We are now in a position to prove one of our crucial theorems. Let $\omega(G)$ denote the clique number of a graph $G$. That is, $G$ contains a clique of order $\omega(G)$ but does not contains a clique of order $\omega(G)+1$. 
\begin{theorem} \label{Thm1}
For any integer $n\geq 2$, we have \[\eta_n=\omega(\widehat H_n).\] 
\end{theorem} 
\begin{proof} 
Choose an integer $n\geq 2$, and let $k=\omega(\widehat H_n)$. Let $C=\{\vec x_1,\vec x_2,\ldots,\vec x_k\}$ be a clique of order $k$ of $\widehat H_n$. By Lemma \ref{Lem2}, there exists a map $f\colon \mathbb F_2^n\rightarrow \mathbb F_2$ such that any distinct vectors $\vec x_i,\vec x_j\in C$ are in distinct $2$-cycles of $\Gamma([K_n,f,\text{id}])$. In particular, $\Gamma([K_n,f,\text{id}])$ contains at least $k$ $2$-cycles. In the notation of Definition \ref{Def3}, $f\in (\mathbb F_2)^{\mathbb F_2^n}$, and $\eta(f,\text{id})\geq k$. Thus, \[\eta_n=\displaystyle{\max_{g\in (\mathbb F_2)^{\mathbb F_2^n}}\eta(g,\text{id})}\geq\eta(f,\text{id})\geq k.\]

We now show that $\eta_n\leq k$. By Definition \ref{Def3}, there exists a function $g\in (\mathbb F_2)^{\mathbb F_2^n}$ such that $\eta(g,\text{id})=\eta_n$. In other words, there are $\eta_n$ $2$-cycles in the phase space $\Gamma([K_n,g,\text{id}])$. It follows from Lemma \ref{Lem8} that each $2$-cycle of $\Gamma([K_n,g,\text{id}])$ contains exactly one vector whose first coordinate is $0$. In the notation of Definition \ref{Def2}, each of the $\eta_n$ $2$-cycles of $\Gamma([K_n,g,\text{id}])$ contains exactly one vector that is a vertex of $\widehat H_n$. Let $\vec x_1,\vec x_2,\ldots,\vec x_{\eta_n}$ be these vertices of $\widehat H_n$. Choose some distinct $i,j\in\{1,2,\ldots,\eta_n\}$. Because each of the vectors $\vec x_i,\vec x_j$ is in a $2$-cycle of $\Gamma([K_n,g,\text{id}])$, it follows from Lemma \ref{Lem3} that $\vec x_i$ and $\vec x_j$ must be adjacent in $\widehat H_n$. Because $i$ and $j$ were arbitrary, this shows that $\{\vec x_1,\vec x_2,\ldots,\vec x_{\eta_n}\}$ is a clique of $\widehat H_n$. Consequently, $\eta_n\leq k$.
\end{proof}

Given an integer $n\geq 2$, we may define a map $\theta_n\colon\widehat{\mathbb F}_2^n\rightarrow\mathbb F_2^{n-1}$ by \[\theta_n(x_1,x_2,\ldots,x_n)=(x_2,x_3,\ldots,x_n).\] Because the first coordinate of each vector in $\widehat{\mathbb F}_2^n$ is $0$, it should be clear that $\theta_n$ is a vector space isomorphism. Furthermore, if $\vec x,\vec y\in\widehat{\mathbb F}_2^n$, then $\vec x+\vec y\in D_n(101)$ if and only if $\theta_n(\vec x)+\theta_n(\vec y)\in D_{n-1}(101)$. This motivates the following definition. 
\begin{definition} \label{Def4} 
Given an integer $m\geq 2$, define $H_m$ to be the undirected simple graph with vertex set \[V(H_m)=\mathbb F_2^m\] and edge set \[E(H_m)=\{\{\vec x,\vec y\}\subseteq\mathbb F_2^m\colon\vec x+\vec y\in D_m(101)\}.\]
\end{definition}
It follows from the preceding paragraph that $\theta_{n+1}$ is a graph isomorphism from the graph $\widehat H_{n+1}$ to the graph $H_n$. Therefore, $\eta_{n+1}=\omega(\widehat H_{n+1})=\omega(H_n)$. In the following section, we relate the graphs $H_n$ to error-correcting binary codes. 
\section{Binary Codes}
In coding theory, a binary code of length $n$ is a subset of $\mathbb F_2^n$. The elements of a code are known as \emph{codewords}. The \emph{Hamming distance} between two codewords $\vec x=(x_1,x_2,\ldots,x_n)$ and $\vec y=(y_1,y_2,\ldots,y_n)$, which we shall denote by $\delta(\vec x,\vec y)$, is simply the number of positions in which the vectors $\vec x$ and $\vec y$ have different coordinates. That is, \[\delta(\vec x,\vec y)=\vert\{i\in\{1,2,\ldots,n\}\colon x_i\neq y_i\}\vert.\] If $C$ is a nonempty binary code of length $n$, then the \emph{minimum distance} of $C$, denoted $\Delta(C)$, is the quantity \[\Delta(C)=\min_{\substack{\vec x,\vec y\in C \\ \vec x\neq \vec y}}\delta(\vec x,\vec y).\]  We make the convention that $\Delta(C)=\infty$ if the code $C$ consists of a single codeword.  

Of particular interest in coding theory are binary codes with minimum distance at least $3$. Such codes are known as one-bit error-correcting codes because any error formed by flipping a single bit (that is, changing a single coordinate) in a codeword can be detected and corrected (this is not the case for binary codes with minimum distance less than $3$). Let $A(n,3)$ denote the maximum number of codewords that can appear in a binary code of length $n$ that has minimum distance at least $3$. An important problem in coding theory is the determination of the values of $A(n,3)$. It is known that $A(n,3)=2^{n-\log_2(n+1)}$ if $n+1$ is a power of $2$, but most values of $A(n,3)$ are not known.

Given an integer $m\geq 2$, define $J_m$ to be the undirected simple graph with vertex set \[V(J_m)=\mathbb F_2^m\] and edge set \[E(J_m)=\{\{\vec x,\vec y\}\subseteq\mathbb F_2^m\colon\vec x+\vec y\in D_m(111)\}.\] Observe that two vectors $\vec x,\vec y\in\mathbb F_2^m$ are adjacent in $J_m$ if and only if the Hamming distance $\delta(\vec x,\vec y)$ between $\vec x$ and $\vec y$ is at least $3$. Therefore, $A(n,3)$ is equal to the clique number $\omega(J_n)$ of the graph $J_n$. From this, we obtain the following theorem. 
\begin{theorem} \label{Thm2} 
For any positive integer $n$, \[\eta_{n+1}=A(n,3).\]
\end{theorem} 
\begin{proof} 
We saw at the end of the preceding section that $\eta_{n+1}=\omega(H_n)$. Therefore, in light of the preceding paragraph, we see that it suffices to show that $\omega(H_n)=\omega(J_n)$. Consider the linear transformation $T\colon\mathbb F_2^n\to\mathbb F_2^n$ given by \[T(x_1,x_2,\ldots,x_n)=(x_1,x_1+x_2,x_1+x_2+x_3,\ldots,x_1+x_2+\cdots+x_n).\] The linear tranformation $T$ is an endomorphism with a trivial kernel, so it is an isomorphism. Furthermore, a vector $\vec z\in\mathbb F_2^n$ contains the subsequence $111$ if and only if $T(\vec z)$ contains the subsequence $101$. This shows that $T$ is in fact a graph isomorphism from $J_n$ to $H_n$, so $\omega(H_n)=\omega(J_n)$. 
\end{proof}

As mentioned in the introduction, Theorem \ref{Thm2} provides the first known interpretation of the numbers $A(n,3)$ outside of coding theory.

\section{Concluding Remarks} 
We have defined $\eta_n$ to be the maximum number of $2$-cycles that can appear in a phase space of the form $\Gamma([K_n,f,\text{id}])$. One could easily generalize these numbers by defining $\eta_n(m)$ to be the maximum number of $m$-cycles that can appear in such a phase space. One may show that $\eta_n(1)=2$ for all positive integers $n$ (the only possible fixed points of an SDS map $[K_n,f,\text{id}]$ are the all-$0$'s and all-$1$'s vectors of length $n$). Is it possible to obtain general bounds for the numbers $\eta_n(m)$? Could we perhaps relate the numbers $\eta_n(m)$ to codes as we have done for the numbers $\eta_n(2)$? 

There are, of course, many other natural ways to generalize the problems considered here. One might wish to replace complete graphs with graphs that are, in some sense, ``almost" complete (such as complements of cycle graphs). We could also choose to ask similar questions about SDS maps of the form $[K_n,f,\text{id}]$ defined using a set of states $A$ with $\vert A\vert\geq 3$. 
\section{Acknowledgments}
The author would like to thank Padraic Bartlett for introducing the author to sequential dynamical systems and giving excellent suggestions for the improvement of this paper. The author would also like to thank the anonymous DMTCS referees for several helpful suggestions.

\end{document}